\documentclass[reqno,b5paper]{amsart}
\usepackage{amsmath}
\usepackage{amssymb}
\usepackage{amsthm}
\usepackage{enumerate}
\usepackage[mathscr]{eucal}

\newtheorem{lemma}{Lemma}[section]

\newtheorem{example}{Example}[section]

\newtheorem{theorem}{Theorem}[section]

\newtheorem{remark}{Remark}[section]

\begin{document}

\title[New characterizations of minimal cusco maps]{New characterizations of minimal cusco  maps}
\author[\v Lubica Hol\'a \and Du\v san Hol\'y]{\v Lubica Hol\'a* \and Du\v san Hol\'y**}

\newcommand{\acr}{\newline\indent}

\address{\llap{*}Academy of Sciences, Institute of Mathematics \acr \v
Stef\'anikova 49,
81473 Bratislava, Slovakia
\acr Slovakia}

\email{hola@mat.savba.sk}

\address{\llap{**\,}Faculty of Industrial Technologies in P\'{u}chov \\
Tren\v{c}\'{\i}n University of Alexander Dub\v{c}ek in
Tren\v{c}{\'{\i}}n
\\I. Krasku 491/30, 02001 P\'{u}chov, Slovakia}

\email{holy@fpt.tnuni.sk}

\thanks{}

\keywords{minimal cusco map, quasicontinuous function, subcontinuous function, set-valued mapping, selection, extreme function. Both authors are thankful to grant APVV-0269-11,  \v L. Hol\'a would like to thank to grant Vega 2/0047/10}

\begin{abstract}
We give new characterizations of minimal cusco maps in the class of all set-valued maps extending results from [BZ1] and [GM]. Let $X$ be a topological space and $Y$ be a Hausdorff locally convex linear topological space. Let $F: X \to Y$ be a set-valued map. The following are equivalent: (1) $F$ is minimal cusco; (2) $F$ has nonempty compact values, there is a quasicontinuous, subcontinuous selection $f$ of $F$ such that $F(x) = \overline{co}\overline f(x)$ for every $x \in X$; (3) $F$ has nonempty compact values, there is a densely defined  subcontinuous, quasicontinuous selection $f$ of $F$ such that $F(x) = \overline{co}$$\overline f(x)$ for every $x \in X$; (4) $F$ has nonempty compact  convex values, $F$ has a closed graph, every extreme function of $F$ is quasicontinuous, subcontinuous and any two extreme functions of $F$ have the same closures of their graphs. Some applications to known results are given.
\end{abstract}

\maketitle

\section{Introduction}
\bigskip
The acronym usco (cusco) stands for a (convex) upper semicontinuous non-empty compact-valued set-valued map. Such set-valued maps are interesting because they describe common features of maximal monotone operators, of the convex subdifferential and of Clarke generalized gradient. Examination of cuscos and uscos leads to serious insights into the underlying topological properties of the convex subdifferential and the Clarke generalized gradient. (It is known that Clarke subdifferential of a locally Lipschitz function and, in particular, the subdifferential of a convex continuous functions are weak* cuscos.) (see [BZ1])

In our paper we are interested in minimal usco and minimal cusco maps. Minimal usco and minimal cusco maps are used in many papers (see [BZ1], [BZ2], [DL], [GM], [HH], [Wa]). We give new characterizations of minimal usco and minimal cusco maps in the class of all set-valued maps using  densely defined subcontinuous quasicontinuous selections. We have also a new characterization of minimal cusco maps using extreme selections. Notice that all known characterizations of minimal usco (cusco) maps are given in the class of usco (cusco) maps (see {BZ1], [GM]). Our approach gives a possibility to construct a minimal usco (cusco) map very easily.

\bigskip
\section{Minimal cusco maps}
\bigskip
In what follows let $X, Y$ be Hausdorff topological spaces, $\Bbb R$ be the space of real numbers with the usual metric and $Z^+$ be the set of positive integers. Also, for $x \in X$, $\mathcal U(x)$ is always used to denote a base of open neighborhoods of $x$ in $X$. The symbol $\overline A$ and $Int A$ will stand for the closure and interior of the set $A$ in a topological space.

A set-valued map, or multifunction, from $X$ to $Y$ is a function that assigns to each element of $X$ a subset of $Y$. If $F$ is a set-valued map from $X$ to $Y$, then its graph is the set  $\{(x,y) \in X \times Y: y \in F(x)\}$. Conversely, if $F$ is a subset of $X \times Y$ and $x \in X$, define $F(x) = \{y \in Y: (x,y) \in F\}$. Then we can assign to each subset $F$ of $X \times Y$ a set-valued map which takes the value $F(x)$ at each point $x \in X$ and which graph is $F$. In this way, we identify set-valued maps with their graphs. Following [DL] the term map is reserved for a set-valued map.

Notice that if $f: X \to Y$  is a single-valued function, we will use the symbol $f$ also for the graph of $f$.

\bigskip
Given two maps $F, G: X \to Y$, we write $G \subset F$ and say that $G$ is contained in $F$ if $G(x) \subset F(x)$ for every $x \in X$.

A map $F: X \to Y$ is upper semicontinuous at a point $x \in X$ if for every open set $V$ containing $F(x)$, there exists $U \in \mathcal U(x)$ such that
\bigskip

\centerline{$F(U) = \cup \{F(u):u \in U\} \subset V.$}
\bigskip

$F$ is upper semicontinuous if it is upper semicontinuous at each point of $X$. Following Christensen [Ch] we say, that a map $F$ is usco if it is upper semicontinuous and takes nonempty compact values. A map $F$ from a topological space $X$ to a linear topological space $Y$ is cusco if it is usco and $F(x)$ is convex for every $x \in X$.

Finally, a map $F$ from a topological space $X$ to a topological (linear topological space) $Y$ is said to be minimal usco (minimal cusco) if it is a minimal element in the family of all usco (cusco) maps (with domain $X$ and range $Y$); that is, if it is usco (cusco) and does not contain properly any other usco (cusco) map from $X$ into $Y$. By an easy application of the Kuratowski-Zorn principle we can guarantee that every usco (cusco)  map from $X$ to $Y$ contains a minimal usco (cusco) map from $X$ to $Y$ (see [BZ1], [BZ2], [DL]).

Other approach to minimality of set-valued maps can be found in [Ma] and [KKM].

In the paper [HH] we can find an interesting characterization of minimal usco maps using quasicontinuous and subcontinuous selections.

A function $f: X \to Y$ is quasicontinuous at $x \in X$ [Ne] if for every neighborhood $V$ of $f(x)$ and every $U \in \mathcal U(x)$ there is a nonempty open set $G \subset U$ such that $f(G) \subset V$. If $f$ is quasicontinuous at every point of $X$, we say that $f$ is quasicontinuous.

The notion of quasicontinuity was perhaps the first time used by R. Baire in [Ba] in the study of points of separately continuous functions. As Baire indicated in his paper [Ba] the condition of quasicontinuity has been suggested by Vito Volterra. There is a rich literature concerning the study of quasicontinuity, see for example [Ba], [Bo], [HP], [Ke], [KKM], [Ne]. A condition under which the pointwise limit of a sequence of quasicontinuous functions is quasicontinuous was studied in [HHo].

A function $f: X  \to Y$ is subcontinuous at $x \in X$ [Fu] if for every net $(x_i)$ convergent to $x$, there is a convergent subnet of $(f(x_i))$. If $f$ is subcontinuous at every $x \in X$, we say that $f$ is subcontinuous.

Let $F:X \to Y$ be a set-valued map. Then a function $f:X \to Y$ is called a selection of $F$ if $f(x) \in F(x)$ for every $x \in X$.

It is well known that every selection of a usco map is subcontinuous ([HH], [HN]).

\begin{theorem} (see Theorem 2.5 in [HH])  Let $X, Y$ be topological spaces and $Y$ be a $T_1$ regular space. Let $F$ be a map from $X$ to $Y$. The following are equivalent:

(1) $F$ is a minimal usco map;

(2) There exist a quasicontinuous and subcontinuous selection $f$ of $F$ such that $\overline f = F$;

(3) Every selection $f$ of $F$ is quasicontinuous, subcontinuous and $\overline f = F$.
\end{theorem}
\bigskip

Let $Y$ be a linear topological space and $B \subset Y$ is a set. By $\overline{co} B$ we denote the closed convex hull of the set $B$ (see [AB]).
\bigskip

The following Lemma is a folklore.

\begin{lemma} Let $X$ be a topological space and $Y$ be a Hausdorff locally convex linear topological space. Let $G$ be a usco map from $X$ to $Y$ and  $\overline{co}$$ G(x)$ is compact for every $x \in X$. Then the map $F$ defined as $F(x) = \overline{co}$$ G(x)$ for every $x \in X$ is a cusco map.
\end{lemma}

\bigskip
\begin{remark} There are three important cases when the closed convex hull of a compact set is compact. The first is when the compact set is a finite union of compact convex sets. The second is when the space is completely metrizable and locally convex. This includes the case of all Banach spaces with their norm topologies. The third case is a compact set in the weak topology on a Banach space. (see[AB])
\end{remark}

A set-valued map $F$ from a topological space $X$ to a linear topological space $Y$ is hyperplane minimal [BZ1] if for every open half-space $W$ in $Y$ and open set $U$ in $X$ with $F(U) \cap W \ne \emptyset$ there is a nonempty open subset $V \subset U$ such that $F(V) \subset W$. It is known [BZ1] that a cusco map from a topological space $X$ into Hausdorff locally convex linear topological space $Y$ is minimal cusco if, and only if, it is hyperplane minimal.

If $f: X \to Y$ is a quasicontinuous function from a topological space to a linear topological space then $f$ is hyperplane minimal. The following example is an example of a hyperplane minimal function which is not quasicontinuous.

\begin{example} Let $X = Y =\Bbb R$ with the usual topology. Define $f: X \to Y$ as follows: $f(x) = -1$ if $x < 0$, $f(0) = 0$ and $f(x) = 1$ if $x > 0$.

\end{example}

\bigskip

Notice that all known characterizations of minimal cusco maps are given in the class of cusco maps (see [GM], [BZ1]). So the following characterization of minimal cusco maps in the class of all set-valued maps can be of some interest:

\begin{theorem} Let $X$ be a topological space and $Y$ be a Hausdorff locally convex (linear topological) space. Let $F$ be a map from $X$ to $Y$. Then the following are equivalent:

(1) $F$ is a minimal cusco map;

(2) $F$ has nonempty compact values and there is a quasicontinuous, subcontinuous selection $f$ of $F$ such that $\overline {co}$$ \overline f(x) = F(x)$ for every $x \in X$;

(3) $F$ has nonempty compact  values and there is a hyperplane minimal, subcontinuous selection $f$ of $F$ such that $\overline{co}$$\overline f(x) = F(x)$ for every $x \in X$;

(4) $F$ has nonempty compact values and every selection $f$ of $F$ is hyperplane minimal, subcontinuous and $\overline{co}$$ \overline f(x) = F(x)$ for every $x \in X$.

\end{theorem}
\begin{proof} $(1) \Rightarrow (2)$ Let $G \subset F$ be a minimal usco map contained in $F$. Let $f$ be a selection of $G$. By Theorem 2.1 $f$ is a quasicontinuous and subcontinuous selection of $G$ such that $\overline f = G$. So $f$ is also a selection  of $F$. By Proposition 2.7 in [BZ1] we have  $\overline{co}$$ \overline f(x) = F(x)$ for every $x \in X$.

$(2) \Rightarrow (3)$ is trivial, since every quasicontinuous function from $X$ to $Y$ is hyperplane minimal.

$(3) \Rightarrow (1)$ Let $f$ be a hyperplane minimal, subcontinuous selection of $F$. Since $f$ is subcontinuous,   $\overline f$ is usco by [HN]. Since $\overline{co}$$\overline f(x) = F(x)$ for every $x \in X$ and $F(x)$ is compact for every $x \in X$, $F$ is cusco by Lemma 2.1.
Thus it is sufficient to show that $F$ is minimal. Suppose, by way of contradiction, that $F$ is not minimal. Thus there is a minimal cusco map $L \subset F$ such that there is a point $(x_0,y_0) \in F \setminus L$. Since $L(x_0)$ is a convex set and $\overline{co}$$\overline f(x_0) = F(x_0)$, without loss of generality we can suppose that $y_0 \in \overline f(x_0) \setminus L(x_0)$. Since $L(x_0)$ is a closed convex set and $y_0 \notin L(x_0)$, there is a nonzero continuous linear functional strongly separating $L(x_0)$ and $y_0$. So let $h: Y \to R$ be a continuous linear functional and $\lambda \in R$ such that

\bigskip
\centerline{$L(x_0) \subset \{y \in Y: h(y) < \lambda\}$ and $h(y_0) > \lambda$.}

\bigskip
Since the map $L$ is upper semicontinuous there is $U \in \mathcal U(x_0)$ such that $L(U) \subset \{y \in Y: h(y) < \lambda\}$ and since $y_0 \in \overline f(x_0)$ and $f$ is hyperplane minimal, there is a nonempty open set $V \subset U$ such that $f(V) \subset \{y \in Y: h(y) > \lambda\}$. Thus $\overline f(V) \subset \{y \in Y: h(y) \ge \lambda\}$. For every $x \in V$ we have $\overline{co}$$\overline f(x) \cap L(x) = \emptyset$, a contradiction.

Since $(4) \Rightarrow (3)$ is trivial, it is sufficient to prove that $(1) \Rightarrow (4)$. Let $f$ be a selection of $F$. Since every selection of a usco map is subcontinuous,  $f$ must be subcontinuous.  $\overline f$ is usco and $\overline f \subset F$ implies that $\overline{co}$$\overline f(x)$ is compact for every $x \in X$. By Lemma 2.1 the map $G$ defined as $G(x) =\overline{co}$$\overline f(x)$ for every $x \in X$ is cusco. Since $G \subset F$ and $F$ is minimal, we have $\overline{co}$$\overline f(x) = F(x)$ for every $x \in X$. It is easy to verify from Theorem 2.6 in [BZ1] that $f$ is hyperplane minimal.

\end{proof}
\bigskip
We have the following variant of Theorem 2.2:

\begin{theorem} Let $X$ be a topological space and $Y$ be a Hausdorff locally convex (linear topological) space in which the closed convex hull of a compact set is compact. Let $F: X \to Y$ be a set-valued map. The following are equivalent:

(1) $F$ is minimal cusco map;

(2) There is a quasicontinuous subcontinuous function $f: X \to Y$ such that $\overline{co}$$\overline f(x) = F(x)$ for every $x \in X$;

(3) There is a hyperplane minimal subcontinuous function $f: X \to Y$ such that $\overline{co}$$ \overline f(x) = F(x)$ for every $x \in X$;

(4) Every selection $f$ of $F$ is hyperplane minimal and subcontinuos and $\overline{co}$$ \overline f(x) = F(x)$ for every $x \in X$.

\end{theorem}
\bigskip
Notice that Theorem 2.3 gives us a rule how to construct minimal cusco maps with values in Hausdorff locally convex (linear topological) spaces in which the closed convex hull of a compact set is compact.

\bigskip
It is interesting to note that our Theorem 2.2 (and also Theorem 2.3) implies the well-known result that every convex function on an open convex subset of a finite dimensional normed linear space is Frechet differentiable on a dense $G_{\delta}$ subset of its domain. Let $f$ be a convex function defined on an open convex subset $A$ of a finite dimensional normed linear space $X$. It is known that the subdifferential mapping $x\rightarrow \partial f(x)$ is a minimal cusco map from $A$ into $X$ [Ph]. Further $f$ is Frechet differentiable at $x \in A$ if and only if the subdifferential mapping $x\rightarrow \partial f(x)$ is single-valued. By Theorem 2.2 (2) there is a quasicontinuous selection $h$ of the subdifferential mapping such that $co \overline h(x)$ = $\partial f(x)$. It is easy to verify that if $x$ is a point of continuity of $h$, then $co \overline h(x) = \{h(x)\}$. It is well known (see for example [HP], [Ne]) that the set of points of continuity of a quasicontinuous function defined on a Baire space with values in a metrizable space is a dense $G_\delta$ set.

\bigskip

In the last part of this section we will extend Theorem 2.18 in [BZ1].
\bigskip

Notice that the notion of subcontinuity can be extend for so-called densely defined functions.

Let $A$ be a dense subset of a topological space $X$ and $Y$ be a topological space. Let $f: A \to Y$  be a function. We say that $f$ is densely defined. Further we say that $f: A \to Y$ is subcontinuous at $x \in X$ [LL] if for every net $(x_i) \subset A$, there is a convergent subnet of $(f(x_i))$. It is easy to verify that (*) $f: A \to Y$ is subcontinuous at $x \in X$ if and only if for every open cover $\mathcal H$ of $Y$ there is a finite subset $\mathcal F$ of $\mathcal H$ and $U \in \mathcal U(x)$ such that $f(U \cap A) \subset \cup \mathcal F$ (a slight modification of Theorem 2.1 in [No]).

We say that $f:A \to Y$ is subcontinuous if it is subcontinuous at every $x \in X$.

\bigskip
First we extend Theorem 2.1 using densely defined selections. Let $X, Y$ be topological spaces and $F: X \to Y$ be a map. We say that a densely defined function $f$ is a densely defined quasicontinuous selection of a set-valued map $F$, if $f(x) \in F(x)$ for every $x \in dom f$, the domain of $f$ and $f: dom f \to Y$ is quasicontinuous with respect to the induced topology on $dom f$.

\begin{theorem} Let $X, Y$ be topological spaces and $Y$ be a $T_1$ regular space. Let $F: X \to Y$ be a map. The following are equivalent:

(1) $F$ is minimal usco;

(2) There is a densely defined quasicontinuous subcontinuous selection $f$ of $F$ such that $\overline f = F$.
\end{theorem}
\begin{proof} $(1) \Rightarrow (2)$ is clear from Theorem 2.1. $(2) \Rightarrow (1)$ Let $f$ be a densely defined quasicontinuous subcontinuous selection of $F$. Thus $dom f$, the domain of $f$ is a dense set in $X$. We show that the subcontinuity of $f$ implies that, $\overline f(x)$ is a nonempty compact set for every $x \in X$. Let $x \in X$. Of course $\overline f(x) \ne \emptyset$. Let $\mathcal H$ be an open cover of $\overline f(x)$. Let $\mathcal H'$ be a refinement of $\mathcal H$ such that for every $H' \in \mathcal H'$ there is $H \in \mathcal H$ with $\overline{H'} \subset H$ and $\overline f(x) \subset \cup \mathcal H'$. For every $y \in Y \setminus \overline f(x)$ let $O_y$ be an open neighborhood of $y$ such that $\overline{O_y} \cap \overline f(x) = \emptyset$. Then the family $\mathcal H' \cup \{O_y: y \in Y \setminus  \overline f(x)\}$ is an open cover of $Y$. By (*) there is $U \in \mathcal U(x)$, $H_1', H_2', ... H_n' \in \mathcal H'$ and a finite indexed set $I$ such that $f(U \cap dom f) \subset \cup \{H_i': i = 1, 2, ... n\} \bigcup \cup \{O_{y_i}: i \in I\}$. Thus

\bigskip
\centerline{$\overline f(x) \subset \overline{f(U \cap dom f)} \subset (\overline{H_1'} \cup \overline{H_2'} \cup ... \cup \overline{H_n'}) \bigcup \cup\{\overline{O_{y_i}}: i \in I\}.$}

\bigskip

Thus $\overline f(x)  \subset H_1 \cup H_2 \cup ... \cup H_n,$ where $H_i \in \mathcal H$ for $i = 1, 2, ... n$.

Now we  will show that $\overline f$ is upper semicontinuous. Suppose there is $x \in X$ such that $\overline f$ is not upper semicontinuous at $x$. Let $V$ be an open set in $Y$ with $\overline f(x) \subset V$ such that for every $U \in \mathcal U(x)$ there are $x_U \in U$ and $y_U \in \overline f(x_U) \setminus V$. The regularity of $Y$ implies that there is an open set $G$ in $Y$ such that $\overline f (x) \subset G \subset \overline G \subset V$. Thus for every $U \in \mathcal U(x)$ we have $(x_U,y_U) \in \overline f \cap (U \times (Y \setminus \overline G))$. For every $U \in \mathcal U(x)$ there is $a_U \in dom f \cap U$ such that $f(a_U) \in Y \setminus \overline G$. Since the net $(a_U)_{U \in \mathcal U(x)}$ converges to $x$, the subcontinuity of $f$ at $x$ implies that there is a cluster point $y \in Y \setminus G$ of the net $(f(a_U))_{U \in \mathcal U(x)}$, a contradiction, since $y \in \overline f(x) \subset G$.

To prove that $\overline f$ is minimal usco, it is sufficient to show (by Theorem 2.1) that every selection $g: X \to Y$ of $\overline f$ is quasicontinuous, since every selection of $\overline f$ is subcontinuous (see Proposition 2.3 in [HH]). Let $g: X \to Y$ be a selection of $\overline f$. Let $x \in X$ and $U \in \mathcal U(x)$ and $V$ be an open neighborhood of $g(x)$. Let $G$ be an open neighborood of $g(x)$ such that $g(x) \in G \subset \overline G \subset V$. Since $(x,g(x)) \in \overline f$, there is $(z,f(z)) \in (U\cap dom f) \times G$. The quasicontinuity of $f$ at $z$ implies that there is a nonempty open set $H$ in $X$ such that $H \cap dom f \subset U \cap dom f$ and  $f(H \cap dom f) \subset G$. The set $H \cap U$ is a nonempty open set contained in $U$ and  $\overline f(H \cap U) \subset \overline G \subset V$. Thus $g(H \cap U) \subset V$.

\end{proof}

\bigskip

\begin{remark} Let $X$ be a Baire space and $F: X \to \Bbb R$ be usco. Let $f: X \to \Bbb R$ be a function defined as follows: $f(x) = inf\{t \in \Bbb R: t \in F(x)\}$ for $x \in X$. Then $f$ is a lower semicontinuous function. It is known (see [En]) that the set $C(f)$ of the points of continuity of $f$ is a dense $G_\delta$ set in $X$. Thus by Theorem 2.4 the map $G = \overline {f\upharpoonright C(f)}$ is a minimal usco map from $X$ to $ \Bbb R$ and $G \subset F$.

Also if $h: X \to \Bbb R$ is defined as $h(x) = sup\{t \in \Bbb R: t \in F(x)\}$ for $x \in X$, then $h$ is upper semicontinuous and by [En] the set $C(h)$ of the points of  continuity of $h$ is a dense $G_\delta$ set in $X$. Thus by Theorem 2.4 the map $H = \overline{h \upharpoonright C(h)}$ is a minimal usco map from $X$ to  $\Bbb R$ and $H \subset F$.

Proposition 5.1.24 in [BZ2] gives a construction of a minimal usco map contained in a given usco map from a general topological space with values in $\Bbb R$.

\end{remark}

\bigskip
We have the following extension of Theorem 2.18 in [BZ1]:
\begin{theorem} Let $X$ be a topological space and $Y$ be a Hausdorff locally convex (linear topological) space. Let $F: X \to Y$ be a map. The following are equivalent:

(1) $F$ is minimal cusco;

(2) $F$ has a nonempty compact values and there is a densely defined  quasicontinuous, subcontinuous  selection $f$ of $F$ such that $ \overline{co}$$\overline f(x) = F(x)$ for every $x \in X$;

(3) $F$ has a nonempty compact values and there is a densely defined hyperplane minimal, subcontinuous selection $f$ of $F$ such that $\overline{co}$$\overline f(x) = F(x)$ for every $x \in X$.

\end{theorem}
\begin{proof} $(1) \Rightarrow (2)$ is clear from the  Theorem 2.2. $(2) \Rightarrow (3)$ is trivial.

$(3) \Rightarrow (1)$ Let $f$ be a densely defined hyperplane minimal, subcontinuous selection of $F$ such that $\overline{co}$$\overline f(x) = F(x)$ for every $x \in X$. As in the above proof we can show that $\overline f$ is usco. Since $F$ has compact values, the map $x\rightarrow \overline{co}$$\overline f(x)$ is cusco (by Lemma 2.1). To prove that $F$ is minimal cusco we can use the same argument as in the proof of $(3) \Rightarrow (1)$ of Theorem 2.2.
\end{proof}
\bigskip
\bigskip
\bigskip
To see that our Theorem 2.5 is an extension of Theorem 2.18 in [BZ1] we need the following comment:

Let $X$ be a topological space and $Y$ be a Hausdorff locally convex (linear topological) space. If $f$ is densely defined  subcontinuous function such that $\overline{co}$$\overline f(x)$ is compact for every $x \in X$, then  $CSC(f)(x) = \overline{co}$$\overline f(x)$ for every $x \in X$, where

\bigskip
\centerline{$CSC(f)(x) = \cap\{\overline{co}$$f(V): V \in \mathcal U(x)\}$ [BZ1].}
\bigskip

Notice that the authors in [BZ1] work in their Theorem 2.18 only with densely defined selections of cusco maps; i.e. with subcontinuous selections $f$ such that $\overline{co}$$\overline f(x)$ is compact for every $x \in X$.

However the condition of subcontinuity of $f$ is essential as the following example shows. (The inclusion $\overline{co}$$\overline f(x) \subset CSC(f)(x)$ can be proper.)

\bigskip
\begin{example} Let $X = \Bbb R = Y$ with the usual topology. Let $f: X \to Y$ be defined as follows: $f(x) = 0$ for every $x \le 0$ and $f(x) = 1/x$ for every $x > 0$. Then $\overline{co}$$\overline f(x) = \{f(x)\}$ for every $x \in X$ and $CSC(f)(0) = [0,\infty)$ and $CSC(f)(x) = \{f(x)\}$ otherwise. Of course $f$ is not subcontinuous at $0$.

\end{example}

\bigskip
\bigskip
\bigskip
\bigskip
\bigskip

We have the following variant of Theorem 2.5:

\begin{theorem} Let $X$ be a topological space and $Y$ be a Hausdorff locally convex (linear topological) space in which the closed convex hull of a compact set is compact. Let $F: X \to Y$ be a map. The following are equivalent:

(1) $F$ is minimal cusco;

(2) There is a densely defined quasicontinuous subcontinuous function $f$ with values in $Y$ such that $\overline{co}$$\overline f(x) = F(x)$ for every $x \in X$;

(3) There is a densely defined hyperplane minimal subcontinuous function $f$ with values in $Y$ such that $\overline{co}$$\overline f(x) = F(x)$ for every $x \in X$.

\end{theorem}
\bigskip

Notice that Theorem 2.14 in [BZ1] is an easy consequence of our Theorem 2.6. The function $f: G \to \Bbb R$ from Lemma 2.13 in [BZ1] is defined on a dense $G_\delta$ set $G$ of a topological space $T$. It is easy to verify that $f$ is subcontinuous. Since $f$ is continuous on $G$, by our Theorem 2.6 the map $x \rightarrow co \overline f(x)$ (for every $x \in T$) is minimal cusco. Of course for $\Phi$ in Theorem 2.14 in [BZ1] we have $\Phi(x) = co \overline f(x)$ for every $x \in T$.
\bigskip
\bigskip
\bigskip

\begin{remark} Let $X$ be a Baire space and $F: X \to  \Bbb R$ be cusco. Let $f: X \to \Bbb R$ be a function defined as $f(x) = inf\{t \in \Bbb R: t \in F(x)\}$ for $x \in X$. Using Remark 2.2 and our Theorem 2.6 we see that the map $x \rightarrow co \overline{f \upharpoonright C(f)}(x)$ is a minimal cusco map contained in $F$.

Similarly, if $h: X \to \Bbb R$ is a function defined as $h(x) = sup\{t \in \Bbb R: t \in F(x)\}$ for $x \in X$ then the map $x \rightarrow co$$\overline{h \upharpoonright C(h)}(x)$ is a minimal cusco map contained in $F$.

\end{remark}
\bigskip
\section{Minimal cusco maps and extreme functions}

\bigskip
Let $B$ be a subset of a linear topological space. By $\mathcal E(B)$ we denote the set of all extreme points of $B$.

Let $X$ be a topological space and $Y$ be a Hausdorff locally convex (linear topological) space. Let $F: X \to Y$ be a map with nonempty compact values. Then a selection $f$ of $F$ such that $f(x) \in \mathcal E(F(x))$ for every $x \in X$ is called an extreme function of $F$. (By Corollary 7.66 in [AB] every nonempty compact subset of a Hausdorff locally convex (linear topological) space has an extreme point. The hypothesis of local convexity cannot be dispensed. [AB], page 298)

\bigskip

\begin{lemma} Let $X$ be a topological space and $Y$ be a Hausdorff locally convex (linear topological) space. Let $F: X \to Y$ be a minimal cusco map and $G: X \to Y$ be a minimal usco map such that $G \subset F$. Then $\mathcal E(F(x)) \subset G(x)$ for every $x \in X$.

\end{lemma}

\begin{proof} Let $x \in X$. By Proposition 2.7 in [BZ1] we have that $F(x) = \overline{co} G(x)$ for every $x \in X$. By Theorem 2.10.15 in [Me] which was proved by D.P. Milman in his paper [Mi]  every extreme point of $\overline{co} G(x)$ is contained in $G(x)$. Thus $\mathcal E(F(x)) \subset G(x)$ for every $x \in X$.

\end{proof}

\bigskip

\begin{theorem} Let $X$ be a topological space and $Y$ be a Hausdorff locally convex  (linear topological)  space. Let $F: X \to Y$ be a map. The following are equivalent:

(1) $F$ is a minimal cusco map;

(2) $F$ has nonempty compact, convex values, $F$ has a closed graph, every extreme function of $F$ is quasicontinuous, subcontinuous and any two extreme functions of $F$ have the same closures of their graphs;

(3) $F$ has nonempty compact values, every extreme function $f$ of $F$ is quasicontinuous, subcontinuous   and $F(x) = \overline{co}$$ \overline f(x)$ for every $x \in X$.

\end{theorem}

\begin{proof} $(1) \Rightarrow (2)$ Of course, $F$ has to have nonempty compact, convex values and $F$ has to have a closed graph. Let $f$ be an extreme function of $F$. We will show that $f$ is quasicontinuous and subcontinuous. Let $G$ be a minimal usco map contained in $F$ (there is a unique minimal usco map contained in $F$ by Theorem 4.1). By Lemma 3.1 we have $\mathcal E(F(x)) \subset G(x)$ for every $x \in X$. Since $f(x) \in \mathcal E(F(x))$ for every $x \in X$,  $f$ is a selection of $G$. By Theorem 2.1 $f$ must be quasicontinuous, subcontinuous and $\overline f = G$. Thus every two extreme functions have to have the same closures of their graphs.

$(2) \Rightarrow (3)$  Let $f$ be an extreme function of $F$. (Such a function exists, for $F(x)$ is a nonempty compact set for every $x \in X$.) Since $f$ is quasicontinuous and subcontinuous, $\overline f$ is a minimal usco map by Theorem 2.1 and $\overline f \subset F$. We claim that $F(x) = \overline{co}$$\overline f(x)$ for every $x \in X$.

Suppose there is $(x,y) \in X \times Y$ such that $y \in F(x) \setminus \overline{co}$$ \overline f(x)$.
Without loss of generality we can suppose that $y \in \mathcal E(F(x))$, since by Krein-Milman theorem a compact convex set is the closed convex hull of its extreme points. Since $y \notin \overline{co}$$ \overline f(x)$, there are two open and disjoint sets $O_1, O_2$ in $Y$ such that

\centerline{$\overline{co}$$ \overline f(x) \subset O_1$    and     $y \in O_2$.}

Let $U \in \mathcal U(x)$ be such that $\overline f(U) \subset O_1$. Let $g$ be an extreme function of $F$ such that $g(x) = y$, a contradiction with the fact that every two extreme functions of $F$ have  the same closures of their graphs.

$(3) \Rightarrow (1)$ To prove that $F$ is a minimal cusco map, let $f$ be an extreme function of $F$. Since $f$ is quasicontinuous and subcontinuous, by Theorem 2.1 $\overline f$ is minimal usco and by Lemma 2.1 and Proposition 2.7 in [BZ] a map $x\rightarrow \overline{co}$$ \overline f(x)$ is minimal cusco. Since $F(x) = \overline{co}$$ \overline f(x)$ for every $x \in X$, we are done.

\end{proof}
\bigskip

Let $F \subset X \times \Bbb R$ such that $F(x)$ is a nonempty bounded set for every $x \in X$. Then there are two real-valued functions $sup F$ and $inf F$ defined on $X$ by $sup F(x) = sup \{t \in \Bbb R: t \in F(x)\}$ and $inf F(x) = inf \{t \in \Bbb R: t \in F(x)\}$.

\bigskip

\begin{theorem} Let $X$ be a topological space and $F: X \to \Bbb R$ be a map. The following are equivalent:

(1) $F$ is a minimal cusco map;

(2) $F$ has nonempty compact, convex values, $F$ has a closed graph, $sup F$ and $inf F$ are quasicontinuous, subcontinuous functions and $\overline{sup F} = \overline{inf F}$;

(3) $F$ has nonempty compact  values, $sup F$ and $inf F$ are quasicontinuous, subcontinuous functions
and $F(x) = co$ $  \overline{sup F}(x) = co$ $\overline{inf F}(x)$.
\end{theorem}

\begin{proof}
$(1) \Rightarrow (2)$ is clear from the above Theorem. $(2) \Rightarrow (3)$ We will prove that $F(x) = co$ $\overline{sup F}(x)$ for every $x \in X$.
Suppose there is $(x,y) \in X \times \Bbb R$ such that $y \in F(x) \setminus co$ $\overline{sup F}(x)$. Let $\epsilon > 0$ be such that
\bigskip

\centerline{$(y - 2\epsilon,y + 2\epsilon) \cap co$ $\overline{sup F}(x) = \emptyset$.}
\bigskip

The upper semicontinuity of $z \rightarrow co$ $\overline{sup F}(z)$ at $x \in X$ implies that there is $U \in \mathcal U(x)$ such that $co$ $\overline{sup F}(z) \subset (y + \epsilon,\infty)$ for every $z \in U$. Since $\overline{sup F} = \overline{inf F}$ and $inf F(x) \le y < y + \epsilon$, we have a contradiction.

$(3) \Rightarrow (1)$ By Theorem 2.1 $\overline{sup F}$ is minimal usco. By Lemma 2.1 and Proposition 2.7 in [BZ1] the map $x \rightarrow co $ $\overline{sup F}(x)$ is minimal cusco, so we are done.

\end{proof}
\bigskip

It is interesting to note that Theorem 2.14 in [BZ1] follows also from our Theorem 3.2. In fact, let $f$ be a function from Lemma 2.13 in [BZ1]. Let $H = \overline f$ be the closure of the graph of $f$. Then for $\Phi$ in Theorem 2.14 we have $\Phi(t) = [inf H(t),sup H(t)]$. Since $\overline{inf H} = \overline f = \overline{sup H}$ and $inf H$,  $sup H$ are quasicontinuous and subcontinuous, $\Phi$ is minimal cusco. It is clear that $H(x) = \{f(x)\}$ at every $x \in G$.
\bigskip
\bigskip
\section{Minimal cusco and minimal usco}
\bigskip
\begin{theorem} Let $X$ be a topological space and $Y$ be a Hausdorff locally convex (linear topological) space. Let $F: X \to Y$ be a minimal cusco map. There is a unique minimal usco map contained in $F$.

\end{theorem}

\begin{proof}  Let $G, H$ be two minimal usco maps contained in $F$. It is sufficient to prove that $G(x) \cap H(x) \ne \emptyset$ for every $x \in X$.
Then a map $L: X \to Y$ defined as $L(x) = G(x) \cap H(x)$ for every $x \in X$ is usco and $L \subset G$, $L \subset H$. Thus $G = L = H$.

Let $\tau$ be a Hausdorff locally convex topology on $Y$ and $\Gamma$ be a system of seminorms on $X$ which generate $\tau$. For every $x \in X$, every $p \in \Gamma$ and $\epsilon > 0$ we denote

\bigskip
\centerline{$S_{p,\epsilon}(x) = \{y \in Y: p(x - y) < \epsilon\}$ and $S_{p,\epsilon}(A) = \cup_{a \in A} S_{p,\epsilon}(a)$.}
\bigskip

Suppose there is $x \in X$ such that $G(x) \cap H(x) = \emptyset$. Since $G(x), H(x)$ are compact sets, there is a seminorm $p$ and $\epsilon > 0$ such that
\bigskip

\centerline{$S_{p,\epsilon}(G(x)) \cap S_{p,\epsilon}(H(x)) = \emptyset.$}
\bigskip

The upper semicontinuity of $G$ and $H$ implies that there is $U \in \mathcal U(x)$ such that

\bigskip
\centerline{$G(z) \subset S_{p,\epsilon}(G(x))$ and $H(z) \subset S_{p,\epsilon}(H(x))$}
\bigskip
for every $z \in U$. Let $g \subset G$ and $h \subset H$ be selections of $G, H$ respectively. The quasicontinuity of $h$ at $x$ implies that there is a nonempty open set $V \subset U$ such that

\bigskip
\centerline{$h(V) \subset S_{p,\epsilon/2}(h(x)) \subset \overline{S_{p,\epsilon/2}(h(x))} \subset S_{p,\epsilon}(h(x))$.}

\bigskip
Thus $\overline h(V) \subset \overline{S_{p,\epsilon/2}(h(x))}$, i.e. $\overline{co}$$\overline h(z) \subset \overline{S_{p,\epsilon/2}(h(x))} \subset S_{p,\epsilon}(h(x)) $ for every $z \in V$. For every $z \in V$ we have $F(z) = \overline{co}$$\overline h(z) \subset S_{p,\epsilon}(h(x))$. Since $g(z) \in F(z)$, we have $g(z) \in S_{p,\epsilon}(h(x))$, a contradiction.
\end{proof}

\bigskip
\bigskip
Acknowledgement. The authors are thankful to M. Sleziak.

\end{document}